\documentclass[10pt]{amsart}

\usepackage{amssymb}
\usepackage{graphicx}

\usepackage{hyperref}

\def\N{{\mathbb{N}}}
\def\Z{{\mathbb{Z}}}

\def\R{{\mathbb{R}}}

\title{Regular completions of $\Z^n$-free groups}

\author{Olga Kharlampovich}
\address{Department of Mathematics and Statistics, Hunter College, CUNY, 695 Park Avenue, New York, NY 10065} 
\email{okharlampovich@gmail.com}

\author{Alexei Myasnikov}
\address{Department of Mathematical Sciences, Stevens Institute of Technology, 1 Castle Point on Hudson, Hoboken, NJ 07030, USA} 
\email{amiasnikov@gmail.com}

\author{Denis Serbin}
\address{Department of Mathematical Sciences, Stevens Institute of Technology, 1 Castle Point on Hudson, Hoboken, NJ 07030, USA} 
\email{d.e.serbin@gmail.com}



\date{}

\newtheorem{example}{Example}
\newtheorem{corollary}{Corollary}

\newtheorem{theorem}{Theorem}
\newtheorem{lemma}{Lemma}

\begin{document}

\begin{abstract}
In the present paper we continue studying regular free group actions on $\Z^n$-trees. We show that every finitely generated $\Z^n$-free group $G$ can be embedded into a finitely generated $\Z^n$-free group $H$ acting regularly on the underlying $\Z^n$-tree (we call $H$ a {\em regular $\Z^n$-completion} of $G$) so that the action of $G$ is preserved. Moreover, if $G$ is effectively represented as a group of $\Z^n$-words then the construction of $H$ is effective and $H$ is also effectively represented as a group of $\Z^n$-words.
\end{abstract}

\maketitle

\tableofcontents

\section{Introduction}
\label{se:intro}

The theory of $\Lambda$-trees (where $\Lambda = \R$) has its origins in the papers by I. Chiswell \cite{Chiswell:1976} and J. Tits \cite{Tits:1977}. In particular, in the latter paper the definition of $\R$-tree was given, while the former one established the fundamental connection between group actions on trees and length functions on groups introduced in 1963 by R. Lyndon (see 
\cite{Lyndon:1963}).  Length functions were introduced in an attempt to axiomatize cancelation arguments in free groups as well as free products with amalgamation and HNN extensions, and to generalize them to a wider class of groups. The main idea was to measure the amount of cancellation in passing to the reduced form of the product of reduced words in a free group and free constructions, and it turned out that the cancelation process could be described by rather simple axioms. Using simple combinatorial techniques, Lyndon described groups with {\em free} $\Z$-valued length functions (such length functions correspond to actions on simplicial trees without fixed points).

\smallskip

Later, in their very influential paper \cite{Morgan_Shalen:1984}, J. Morgan and P. Shalen linked group actions on $\R$-trees with topology and generalized parts of Thurston's Geometrization Theorem. Next, they introduced $\Lambda$-trees for an arbitrary ordered abelian group $\Lambda$ and the general form of Chiswell's construction. Thus, it became clear that abstract length functions with values in $\Lambda$ and group actions on $\Lambda$-trees are just two equivalent approaches to the same realm of group theory questions. The unified theory was further developed in the important paper by R. Alperin and H. Bass \cite{Alperin_Bass:1987}, where they state a fundamental problem in the theory of group actions on $\Lambda$-trees: find the group-theoretic information carried by an action on a $\Lambda$-tree (analogous to Bass-Serre theory), in particular, describe finitely generated groups acting freely on $\Lambda$-trees (so called $\Lambda$-free groups). One of the main breakthroughs in this direction is Rips' Theorem, that describes finitely generated $\R$-free groups (see \cite{GLP:1994, Bestvina_Feighn:1995}). The structure of finitely generated $\Z^n$-free groups can be deduced from \cite{Bass:1991} using Lyndon's results (see \cite{Lyndon:1963}) and inductive argument on $n$, while the structure of $\R^n$-free groups was clarified in \cite{Guirardel:2004} using ideas of \cite{Bestvina_Feighn:1995} and again induction on $n$.

\smallskip

Introduction of infinite $\Lambda$-words was one of the major recent developments in the theory of $\Lambda$-free groups. In \cite{Myasnikov_Remeslennikov_Serbin:2005} A. Myasnikov, V. Remeslennikov and D. Serbin showed that groups admitting faithful representations by $\Lambda$-words act freely on $\Lambda$-trees, while Chiswell proved the converse \cite{Chiswell:2005}. This gives another equivalent approach to the whole theory so that one can replace the axiomatic viewpoint of length functions along with many geometric arguments coming from $\Lambda$-trees by combinatorics of $\Lambda$-words. In particular, this approach 
allows one to naturally generalize powerful techniques such as Nielsen's method, Stallings' graph approach to subgroups, and Makanin-Razborov type of elimination processes from free groups to $\Lambda$-free groups (see \cite{Myasnikov_Remeslennikov_Serbin:2005, Myasnikov_Remeslennikov_Serbin:2006, Kharlampovich_Myasnikov:2005b, Kharlampovich_Myasnikov:2006, Kharlampovich_Myasnikov:2010, KMRS:2012, DM, KMS:2011a, KMS:2011b, Nikolaev_Serbin:2011a, Nikolaev_Serbin:2011b, Serbin_Ushakov:2011a}). In the case when $\Lambda$ is equal to either $\Z^n$ or $\Z^\infty$ all these techniques are effective, so, many algorithmic problems for $\Z^n$-free groups become decidable.

\smallskip

While studying $\Lambda$-free groups it becomes evident that it is necessary to introduce some natural restrictions on the action which could significantly simplify many arguments. Thus, given a group $G$ acting on a $\Lambda$-tree $\Gamma$, we say that the action is {\em regular with respect to $x \in \Gamma$} (see \cite{KMRS:2012} for details) if for any $g,h \in G$ there exists $f \in G$ such that $[x, f x] = [x, g x] \cap [x, h x]$. In fact, the definition above does not depend on $x$ and there exist equivalent 
formulations for length functions and $\Lambda$-words (see \cite{Promislow:1985, Myasnikov_Remeslennikov_Serbin:2005}). Roughly speaking, regularity of action implies that all branch-points of $\Gamma$ belong to the same $G$-orbit and it tells a lot about the structure of $G$ in the case of free actions (see \cite{KMS:2011b, KMRS:2012}). Now, given a finitely generated group $G$ acting freely on a $\Lambda$-tree $\Gamma$, several natural question arise: 
\begin{itemize}
\item When does $G$ admit a regular action on $\Gamma$? 
\item Is it possible to change the action of $G$ on $\Gamma$ in order to make it regular?
\item Is it possible to embed $G$ into a finitely generated $\Lambda$-free group $H$ which admits a regular action? Can one do it in an equivariant manner (in this case we call such $H$ a {\em regular $\Lambda$-completion} of $G$)?
\end{itemize}
In particular, the last question has a positive answer (see \cite{Lyndon:1963, Hoare:1979}) in the case when $\Lambda = \Z$ with $H$ being finitely generated (the construction of $H$ is effective in this case). The general case is approached in \cite{Chiswell_Muller:2010}, where the group $H$ is constructed but it is almost never finitely generated (even when $G$ is a finitely generated $\Z^n$-free group). 

In this paper we answer the third question above affirmatively and show that a $\Z^n$-completion of $G$ can be found effectively if one starts with an effective representation of $G$ by infinite words. In particular, the following theorem is proved.

{\bf Theorem 4.} {\em Let $G$ be a finitely generated subgroup of $CDR(\Z^n, X)$, where $X$ is arbitrary. Then there exists a finite alphabet $Y$ and an embedding $\phi : G \rightarrow H$, where $H$ is a finitely generated subgroup of $CDR(\Z^n, Y)$ with a regular length function, such that $|g|_X = |\phi(g)|_Y$ for every $g \in G$. Moreover, if $G$ has an effective hierarchy over $X$, then $H$ has an effective hierarchy over $Y$.}

\section{Preliminaries}
\label{sec:prelim}

Here we introduce the basics of the theory of $\Lambda$-trees (all the details can be found in \cite{Alperin_Bass:1987} and \cite{Chiswell:2001}), Lyndon length functions (see \cite{Lyndon:1963, Chiswell:1976}) and infinite words (see \cite{Myasnikov_Remeslennikov_Serbin:2005}).

\subsection{$\Lambda$-trees}
\label{subs:lambda_trees}

Let $\Lambda$ be an ordered abelian group (we refer the reader to the books \cite{Glass:1999} and \cite{Kopytov_Medvedev:1996}
regarding the general theory of ordered abelian groups) and $X$ a non-empty set. If a function $d: X \times X \to \Lambda$ satisfies the axioms of metric with $\R$ replaced by $\Lambda$, that is, for all $x, y, z \in X$
\begin{enumerate}
\item[(M1)] $d(x, y) \geqslant 0$,
\item[(M2)] $d(x, y) = 0$ if and only if $x = y$,
\item[(M3)] $d(x, y) = d(y, x)$,
\item[(M4)] $d(x, y) \leqslant d(x, z) + d(y, z)$,
\end{enumerate}
then the pair $(X, d)$ is called a  {\em $\Lambda$-metric space}.

The simplest example of a $\Lambda$-metric space is $\Lambda$ itself with the metric $d(a, b) = |a - b|$, for all $a, b \in \Lambda$. 

If $a, b \in \Lambda$ are such that $a \leqslant b$, define $[a, b]_\Lambda = \{x \in \Lambda \mid a \leqslant x \leqslant b\}$.

As usual, if $(X, d)$ and $(X', d')$ are $\Lambda$-metric spaces, an {\em isometry} from $(X, d)$ to $(X', d')$ is a mapping $f: X \rightarrow X'$ such that $d(x, y) = d'(f(x), f(y))$ for all $x, y \in X$. Thus, for $x, y \in X$, a {\em segment} $[x, y]$ in a $X$ is the image of an isometry $\alpha : [a, b]_\Lambda \rightarrow X$ for some $a, b \in \Lambda$ such that $\alpha(a) = x$ and $\alpha(b) = y$ (observe that $[x, y]$ is not unique in general). We call a $\Lambda$-metric space $(X, d)$ {\em geodesic} if at least one segment $[x, y]$ exists for all $x, y \in X$, and $(X, d)$ is {\em geodesically linear} if $[x, y]$ is unique for all $x, y \in X$.

Now, a $\Lambda$-metric space $(X, d)$ is called a {\em $\Lambda$-tree} if
\begin{enumerate}
\item[(T1)] $(X, d)$ is geodesic,
\item[(T2)] if two segments of $(X, d)$ intersect in a single point that is an endpoint of both, then their union is a segment,
\item[(T3)] the intersection of two segments with a common endpoint is also a segment.
\end{enumerate}

\begin{example}
\label{ex:lambda}
$\Lambda$ together with the metric $d(a, b) = |a - b|$ is a $\Lambda$-tree.
\end{example}

\begin{example}
\label{ex:lambda=Z}
A $\Z$-metric space $(X, d)$ is a $\Z$-tree if and only if there is a simplicial tree $\Gamma$ such that $X = V(\Gamma)$ and $d$ is the path metric in $\Gamma$.
\end{example}

We say that a group {\em $G$ acts on a $\Lambda$-tree} $X$ if any element $g \in G$ defines an isometry $g : X \rightarrow X$. $G$ acts on $X$ {\em freely} and {\em without inversions} if no non-trivial $g \in G$ stabilizes a segment in $X$ (a segment can be degenerate). In this case we say that $G$ is {\em $\Lambda$-free}. The action is {\em regular with respect to $x \in X$} if for any $g, h \in G$ there exists $f \in G$ such that $[x, f x] = [x, g x] \cap [x, h x]$ (see \cite{KMS:2011a}).

Given an action of $G$ on a $\Lambda$-tree $(X, d)$, for a point $x \in X$ one can define a function $l_x : G \to \Lambda$ by $l_x(g) = d(x, g x)$. Such a function is called a {\em based length function on $G$} and it is easy to check that $l_x$ satisfies the axioms
\begin{enumerate}
\item[(L1)] $\forall\ g \in G:\ l_x(g) \geqslant 0$ and $l_x(1) = 0$,
\item[(L2)] $\forall\ g \in G:\ l_x(g) = l_x(g^{-1})$,
\item[(L3)] $\forall\ g, f, h \in G:\ c_x(g,f) > c_x(g,h) \rightarrow c_x(g,h) = c_x(f,h)$,

\noindent where $c_x(g,f) = \frac{1}{2}(l_x(g) + l_x(f) - l_x(g^{-1} f))$.
\end{enumerate}
Moreover, if $G$ is $\Lambda$-free then
\begin{enumerate}
\item[(L4)] $\forall\ g \in G:\ l_x(g^2) > l_x(g)$,
\end{enumerate}
and regularity of the action implies the axiom
\begin{enumerate}
\item[(R)] $\forall\ g, h \in G,\ \exists\ u, g_1, h_1 \in G:$
$$g = f \circ g_1 \ \& \ h = f \circ h_1 \ \& \ l_x(f) = c_x(g,h),$$
\end{enumerate}
where $v w = v \circ w$ means that $l_x(v w) = l_x(v) + l_x(w)$.

\smallskip

Now, one can consider an abstract function $l : G \to \Lambda$ with the axioms (L1)--(L3), it is called a {\em Lyndon length function} on $G$. One can show that for such a function $l$ there exists a $\Lambda$-tree $(X, d)$ and a point $x \in X$ such that $l = l_x$ provided $c(g, f) = \frac{1}{2}(l(g) + l(f) - l(g^{-1} f)) \in \Lambda$ for all $f, g \in G$ (see, for example \cite[Theorem 2.4.6]{Chiswell:2001}). $l$ is called {\em free} if it satisfies the axiom (L4) and {\em regular} if it satisfies the axiom (R).

In this paper we are mostly interested in groups with regular free Lyndon length functions and below are several examples.

\begin{example}
\label{ex:0}
Let $F = F(X)$ be a free group on $X$. The length function
$$|\cdot | : F \rightarrow \mathbb{Z},$$
where $|f|$ is a natural length of $f \in F$ as a finite word, is regular since the common initial subword of any two elements of $F$ always exists and belongs to $F$.
\end{example}

\begin{example}
\label{ex:2}
In \cite{Myasnikov_Remeslennikov_Serbin:2005} it was proved that Lyndon's free $\Z[t]$-group has a regular free length function with values in $\Z[t]$.
\end{example}

\begin{example}
\label{ex:3} \cite{KMRS:2012}
Let $F = F(X)$ be a free group on $X$. Consider an HNN-extension
$$G = \langle F, s \mid u^s = v \rangle,$$
where $u, v \in F$ are such that $|u| = |v|$ and $u$ is not conjugate to $v^{-1}$. Then there is a regular free length function $l : G \rightarrow \Z^2$ which extends the natural integer-valued length function on $F$.
\end{example}

For more involved examples, we refer the reader to \cite{KMRS:2012}.

\subsection{Infinite words}
\label{subs:words}

Let $\Lambda$ be an ordered abelian group. $\Lambda$ is called {\em discretely ordered} if it has a minimal positive element. Let us fix a discretely ordered $\Lambda$ for the rest of this subsection. Hence, with a slight abuse of notation we denote the minimal positive element of $\Lambda$ by $1$ and the segment $[a, b]_\Lambda$ for $a, b \in \Lambda$ simply by $[a, b]$.

Now, following \cite{Myasnikov_Remeslennikov_Serbin:2005}, given a set $X = \{x_i \mid i \in I\}$, we put $X^{-1} = \{x_i^{-1} \mid i
\in I\},\ X^\pm = X \cup X^{-1}$, and define a {\em $\Lambda$-word} as a function of the type
$$w: [1, \alpha_w] \to X^\pm,$$
where $\alpha_w \in \Lambda,\ \alpha_w \geqslant 0$. The element $\alpha_w$ is called the {\em length} $|w|_X$ of $w$. Usually we simply write $|w|$ rather than $|w|_X$ when it is clear which underlying alphabet is meant.

In particular, $\Z$-words are finite words in ordinary sense.

Below we refer to $\Lambda$-words as {\em infinite words} usually omitting $\Lambda$ whenever it does not produce any ambiguity.

By $W(\Lambda, X)$ we denote the set of all infinite words. Observe, that $W(\Lambda, X)$ contains an empty word which we denote by $\varepsilon$. Operations of concatenation and inversion are defined on $W(\Lambda, X)$ in the usual way (see \cite{Myasnikov_Remeslennikov_Serbin:2005}).

An infinite word $w$ is {\em reduced} if it does not contain $x x^{-1},\ x \in X^\pm$ as a subword and we denote by $R(\Lambda, X)$ the set of all reduced infinite words. Clearly, $\varepsilon \in R(\Lambda, X)$. If the concatenation $u v$ of two reduced infinite words $u$ and $v$ is also reduced then we write $u v = u \circ v$.

For $u \in W(\Lambda, X)$ and $\beta \in [1, |u|]$ by $u_\beta$ we denote the restriction of $u$ on $[1, \beta]$. If $u \in R(\Lambda, X)$ and $\beta \in [1, |u|]$ then
$$u = u_\beta \circ {\tilde u}_\beta,$$
for some uniquely defined ${\tilde u}_\beta$.

An element $com(u,v) \in R(\Lambda, X)$ is called the (\emph{longest}) {\em common initial segment} of reduced infinite words $u$ and $v$ if
$$u = com(u,v) \circ \tilde{u}, \ \ v = com(u,v) \circ \tilde{v}$$
for some (uniquely defined) infinite words $\tilde{u}$, $\tilde{v}$ such that $\tilde{u}(1) \neq \tilde{v}(1)$. Note that $com(u, v)$ does not always exist.

Now, let $u, v \in R(\Lambda, X)$. If $com(u^{-1}, v)$ exists then
$$u^{-1} = com(u^{-1}, v) \circ {\tilde u}, \ \ v = com(u^{-1}, v) \circ {\tilde v},$$
for some uniquely defined ${\tilde u}$ and ${\tilde v}$. In this event put
$$u \ast v = {\tilde u}^{-1} \circ {\tilde v}.$$
The  product ${\ast}$ is a partial binary operation on $R(\Lambda, X)$.

\smallskip

An element $v \in R(\Lambda, X)$ is termed {\em cyclically reduced} if $v(1)^{-1} \neq v(|v|)$. We say that an element $v \in R(\Lambda, X)$ admits a {\em cyclic decomposition} if $v = c^{-1} \circ u \circ c$, where $c, u \in R(\Lambda, X)$ and $u$ is cyclically reduced. Observe that a cyclic decomposition is unique (whenever it exists). We denote by $CDR(\Lambda, X)$ the set of all words from $R(\Lambda, X)$ that admit a cyclic decomposition.

Now we consider {\em subgroups of $CDR(\Lambda, X)$}, that is, subsets of $CDR(\Lambda, X)$ closed with respect to $\ast$ and inversion of infinite words.

\begin{theorem} \cite{Myasnikov_Remeslennikov_Serbin:2005}
\label{co:3.1}
Any subgroup $G$ of $CDR(\Lambda, X)$ is a group with a free Lyndon length function $| \cdot | : G \to \Lambda$, where $|g|$ is the length of $g$ viewed as an element of $CDR(\Lambda, X)$.
\end{theorem}

The converse is also true.

\begin{theorem} \cite{Chiswell:2005}
\label{chis}
Let $G$ have a free Lyndon length function $l : G \rightarrow \Lambda$, then there exists an embedding $\phi : G \rightarrow CDR(\Lambda, X)$ such that, $|\phi(g)| = l(g)$ for any $g \in G$.
\end{theorem}

Moreover, it was shown in \cite{Khan_Myasnikov_Serbin:2007} that the embedding $\phi$ in Theorem \ref{chis} preserves regularity. Observe that regularity of the length function $| \cdot | $ on a subgroup $H$ of $CDR(\Lambda, X)$ means that $com(g, h) \in H$ for all $g, h \in H$.

\smallskip

Thus, $\Lambda$-free groups are precisely groups with free $\Lambda$-valued Lyndon length functions, which are precisely subgroups of $CDR(\Lambda, X)$ for an appropriate $X$. Given a $\Lambda$-free group $G$, usually we use $\cdot$ to denote the group operation when we view $G$ as an abstract group. At the same time, we can also view $G$ as a subgroup of $CDR(\Lambda, X)$, where elements of $G$ are represented by infinite words, so in this case we can use $\ast$ instead of $\cdot$. The same logic applies to subgroups of $CDR(\Lambda, X)$: we interchangeably use both $\ast$ and $\cdot$ to denote the group operation.

\subsection{Universal trees}
\label{subs:universal}

Let $G$ be a subgroup of $CDR(\Lambda, X)$ for some discretely ordered abelian group $\Lambda$ and a set $X$.

Briefly recall (see \cite{KMS:2011a} for details) how one can construct a universal $\Lambda$-tree $\Gamma_G$ for $G$. Every element $g \in G$ is a function
$$g: [1, |g|] \rightarrow X^{\pm},$$
with the domain $[1, |g|]$, which is a closed segment in $\Lambda$. Since $\Lambda$ can be viewed as a $\Lambda$-metric space, $[1,|g|]$ is a geodesic connecting $1$ and $|g|$, and every $\alpha \in [1,|g|]$ can be viewed as a pair $(\alpha, g)$. Let
$$S_G = \{(\alpha,g) \mid g \in G, \alpha \in [0,|g|]\}.$$
Since for every $f, g \in G$ the word $com(f, g)$ is defined, one can introduce an equivalence relation on $S_G$ as follows: $(\alpha, f) \sim (\beta,g)$ if and only if $\alpha = \beta \in [0, c(f, g)]$. Now, let $\Gamma_G = S_G / \sim$ and $\varepsilon = \langle 0, 1 \rangle$, where $\langle \alpha, f \rangle$ is the equivalence class of $(\alpha, f)$. It was shown in \cite{KMS:2011a} that $\Gamma_G$ is a $\Lambda$-tree with a designated vertex $\varepsilon$ and a metric $d : \Gamma_G \times \Gamma_G \rightarrow \Lambda$, on which $G$ acts by isometries so that for every $g \in G$ the distance $d(\varepsilon, g \cdot \varepsilon)$ is exactly $|g|$. Moreover, $\Gamma_G$ is equipped with the labeling function $\xi : (\Gamma_G - \{\varepsilon\}) \rightarrow X^\pm$, where $\xi(v) = g(\alpha)$ if $v = \langle \alpha, g \rangle$.

\smallskip

It is easy to see that the labeling $\xi$ is not equivariant, that is, $\xi(v) \neq \xi(g \cdot v)$ in general (even if both $v$ and $g \cdot v$ are in $\Gamma_G - \{\varepsilon\}$, which is not stable under the action of $G$). In the present paper we are going to introduce another labeling function for $\Gamma_G$ defined not on vertices but on ``edges'', stable under the action of $G$. With this new labeling $\Gamma_G$ becomes an extremely useful combinatorial object in the case $\Lambda = \Z^n$, but in general such a labeling can be defined for every discretely ordered $\Lambda$.

First of all, for every $v_0, v_1 \in \Gamma_G$ such that $d(v_0, v_1) = 1$ we call the ordered pair $(v_0, v_1)$ the {\em edge} from $v_0$ to $v_1$. Here, if $e = (v_0,v_1)$ then denote $v_0 = o(e),\ v_1 = t(e)$ which are respectively the {\em origin} and {\em terminus} of $e$. Now, if the vertex $v_1 \in \Gamma_G - \{\varepsilon\}$ is fixed then, since $\Gamma_G$ is a $\Lambda$-tree, there is exactly one point $v_0$ such that $d(\varepsilon, v_1) = d(\varepsilon, v_0) + 1$. Hence, there exists a natural orientation, with respect to $\varepsilon$, of edges in $\Gamma_G$, where an edge $(v_0, v_1)$ is {\em positive} if $d(\varepsilon, v_1) = d(\varepsilon, v_0) + 1$, and {\em negative} otherwise. Denote by $E(\Gamma_G)$ the set of edges in $\Gamma_G$. If $e \in E(\Gamma_G)$ and $e = (v_0, v_1)$ then the pair $(v_1, v_0)$ is also an edge and denote $e^{-1} = (v_1, v_0)$. Obviously, $o(e) = t(e^{-1})$. Because of the orientation, we have a natural splitting
$$E(\Gamma_G) = E(\Gamma_G)^+ \cup E(\Gamma_G)^-,$$
where $E(\Gamma_G)^+$ and $E(\Gamma_G)^-$ denote respectively the sets of positive and negative edges. Now, we can define a function $\mu : E(\Gamma_G)^+ \rightarrow X^\pm$ as follows: if $e = (v_0, v_1) \in E(\Gamma_G)^+$ then $\mu(e) = \xi(v_1)$. Next, $\mu$ can be extended to $E(\Gamma_G)^-$ (and hence to $E(\Gamma_G)$) by setting $\mu(f) = \mu(f^{-1})^{-1}$ for every $f \in E(\Gamma_G)^-$.

\begin{example}
\label{ex:1}
Let $F = F(X)$ be a free group on $X$. Hence, $F$ embeds into (in fact, coincides with) $CDR(\Z, X)$ and $\Gamma_F$ with the labeling $\mu$ defined above is just a Cayley graph of $F$ with respect to $X$. That is, $\Gamma_F$ is a labeled simplicial tree.
\end{example}

The action of $G$ on $\Gamma_G$ induces an action on $E(\Gamma_G)$ as follows: $g \cdot (v_0, v_1) = (g \cdot v_0, g \cdot v_1)$ for each $g \in G$ and $(v_0, v_1) \in E(\Gamma_G)$. It is easy to see that $E(\Gamma_G)^+$ is not closed under the action of $G$ but the labeling is equivariant as the following lemma shows (see also \cite[Lemma 3]{KMS:2011a}).

\begin{lemma}
\label{le:label_edges}
If $e, f \in E(\Gamma_G)$ belong to one $G$-orbit then $\mu(e) = \mu(f)$.
\end{lemma}
\begin{proof} 
Let $e = (v_0,v_1) \in E(\Gamma_G)^+$. Hence, there exists $g \in G$ such that $v_0 = \langle \alpha, g \rangle,\ v_1 = \langle \alpha + 1, g \rangle$. Let $f \in G$ and consider the following cases.

\smallskip

\noindent {\bf Case 1}. $c(f^{-1},g) = 0$

Then $f \ast g = f \circ g$. If $\alpha = 0$ then $f \cdot v_0 = \langle |f|, f \rangle = \langle |f|, f \circ g \rangle$, and $f \cdot v_1 = \langle |f| + 1, f \circ g \rangle$. Hence, $f \cdot e \in E(\Gamma_G)^+$ and $\mu(f \cdot e) = \xi(f \cdot v_1) = g(1) = \xi(v_1) = \mu(e)$.

\smallskip

\noindent {\bf Case 2}. $c(f^{-1},g) > 0$

\begin{enumerate}

\item[(a)] $\alpha + 1 \leqslant c(f^{-1},g)$

Then $f \cdot v_0 = \langle |f| + \alpha - 2\alpha, f \rangle = \langle |f| - \alpha, f \rangle$ and $f \cdot v_1 = \langle |f| - (\alpha + 1), f \rangle$. So, $d(\varepsilon, f \cdot v_1) < d(\varepsilon, f \cdot v_0)$ and $f \cdot e \in E(\Gamma_G)^-$. Now,
$$\mu(f \cdot e) = \mu((f \cdot e)^{-1})^{-1} = \mu((f \cdot v_1, f \cdot v_0))^{-1} = \xi(f \cdot v_0)^{-1} = f(|f|-\alpha)^{-1}$$
$$ = g(\alpha+1) = \xi(v_1) = \mu(e).$$

\item[(b)] $\alpha = c(f^{-1},g)$

We have $f \cdot v_0 = \langle |f| - \alpha, f \rangle$ and $f \cdot v_1 = \langle |f| + (\alpha + 1) - 2 c(f^{-1},g), f \ast g \rangle = \langle |f| - \alpha + 1, f \ast g \rangle$. It follows that $f \cdot e \in E(\Gamma_G)^+$ and $\mu(f \cdot e) = \xi(f \cdot v_1) = (f \ast g)(|f| - \alpha + 1)$. At the same time, $f \ast g = f_1 \circ g_1$, where $|f_1| = |f| - c(f^{-1},g) = |f| - \alpha, \ g = g_0 \circ g_1,\ |g_0| = \alpha$, so, $(f \ast g)(|f| - \alpha + 1) = g_1(1) = g(\alpha + 1)$ and $\mu(f \cdot e) = g(\alpha + 1) = \xi(\langle \alpha + 1, g \rangle) = \xi(v_1) = \mu(e)$.

\item[(c)] $\alpha > c(f^{-1},g)$

Hence, $f \cdot v_0 = \langle |f| + \alpha - 2c(f^{-1},g), f \ast g \rangle$ and $f \cdot v_1 = \langle |f| + \alpha + 1 - 2c(f^{-1},g), f \ast g \rangle$. Obviously, $f \cdot e \in E(\Gamma_G)^+$ and
$$\mu(f \cdot e) = \xi(f \cdot v_1) = (f \ast g)(|f| + \alpha + 1 - 2c(f^{-1},g)) = g_1(\alpha + 1 - c(f^{-1},g))$$
$$ = g(\alpha + 1) = \xi(v_1) = \mu(e),$$
where $f \ast g = f_1 \circ g_1,\ |f_1| = |f| - c(f^{-1},g) = |f| - \alpha,\ g = g_0 \circ g_1,\ |g_0| = \alpha$.
\end{enumerate}

Thus, in all possible cases we got $\mu(f \cdot e) = \mu(e)$ and the required statement follows.
\end{proof}

Let $v,w$ be two points of $\Gamma_G$. Since $\Gamma_G$ is a $\Lambda$-tree there exists a unique geodesic connecting $v$ to $w$, which can be viewed as a ``path'' is the following sense. A {\em path from $v$ to $w$} is a sequence of edges $p = \{ e_\alpha \},\ \alpha \in [1, d(v,w)]$ such that $o(e_1) = v,\ t(e_{d(v,w)}) = w$ and $t(e_\alpha) = o(e_{\alpha + 1})$ for every $\alpha \in [1, d(v,w) - 1]$. In other words, a path is an ``edge'' counterpart of a geodesic and usually, for the path from $v$ to $w$ (which is unique since $\Gamma_G$ is a $\Lambda$-tree) we are going to use the same notation as for the geodesic between these points, that is, $p = [v, w]$. In the case when $v = w$ the path $p$ is empty. The {\em length} of $p$ we denote by $|p|$ and set $|p| = d(v,w)$. Now, the {\em path label} $\mu(p)$ of a path $p = \{ e_\alpha \}$ is the function $\mu : \{ e_\alpha \} \rightarrow X^\pm$, where $\mu(e_\alpha)$ is the label of the edge $e_\alpha$.

\begin{lemma} \cite[Lemma 4]{KMS:2011a}
\label{le:paths}
Let $v,w$ be points of $\Gamma_G$ and $p$ the path from $v$ to $w$. Then $\mu(p) \in R(\Lambda, X)$.
\end{lemma}

As usual, if $p$ is a path from $v$ to $w$ then its {\em inverse} denoted $p^{-1}$ is a path from $w$ back to $v$. In this case, the label of $p^{-1}$ is $\mu(p)^{-1}$, which is again an element of $R(\Lambda, X)$.

Define
$$V_G = \{ v \in \Gamma_G \mid\ \exists\ g \in G:\ v = \langle |g|,g \rangle \},$$
which is a subset of points in $\Gamma_G$ corresponding to the elements of $G$. Also, for every $v \in \Gamma_G$ let
$$path_G(v) = \{ \mu(p) \mid\ p = [v,w]\ {\rm where}\ w \in V_G\}.$$
The following lemma follows immediately.

\begin{lemma}
\label{le:paths_2}
Let $v \in V_G$. Then $path_G(v) = G \subset CDR(\Lambda, X)$.
\end{lemma}

The action of $G$ on $E(\Gamma_G)$ extends to the action on all paths in $\Gamma_G$, hence, Lemma \ref{le:label_edges} extends to the case when $e$ and $f$ are two $G$-equivalent paths in $\Gamma_G$.

\section{Effective representation by infinite words}
\label{se:effect}

In this section we recall the notion of effective representation of a group by infinite words originally defined in \cite{KMS:2011a} and then introduce effective hierarchies for $\Z^n$-free groups.

\subsection{Infinite words viewed as computable functions}
\label{subs:comp_func}

Recall (see \cite{KMS:2011a} for details) that a group $G = \langle y_1, \ldots, y_m \rangle$ is said to have an {\em effective representation by $\Lambda$-words over an alphabet $X$} if $G \subset CDR(\Lambda, X)$ and
\begin{enumerate}
\item[(ER1)] for every $i \in [1, m]$, the $\Lambda$-word $y_i$, viewed as the function $y_i : [1, |y_i|] \to X^\pm$, is computable, that is, one can effectively determine $y_i(\alpha)$ for every $\alpha \in [1, |y_i|]$ and $i \in [1,m]$,
\item[(ER2)] for every $i, j \in [1, m]$ and every $\alpha_i \in [1, |y_i|],\ \alpha_j \in [1, |y_j|]$, one can effectively compute $c(h_i, h_j)$, where $h_i = y_i^{\pm 1} \mid_{[\alpha_i, |y_i|]},\ h_j = y_j^{\pm 1} \mid_{[\alpha_j, |y_j|]}$.
\end{enumerate}

Now, suppose $G = \langle Y \rangle$, where $Y = \{y_1, \ldots, y_m\}$, has an effective representation by $\Lambda$-words over $X$. 

Since every $y_i$ is computable, it follows that $y_i^{-1}$ is also computable for every $i \in [1,m]$. Next, the concatenation of two computable $\Lambda$-words is computable, as well as a restriction of a computable function to a computable domain. Thus, if $g_i \ast g_j = h_i \circ h_j$, where $g_i = y_i^{\delta_i} = h_i \circ c,\ g_j = y_j^{\delta_j} = c^{-1} \circ h_j,\ \delta_i, \delta_j = \pm 1$, then both $h_i$ and $h_j$ are computable as the restrictions $h_i = g_i \mid_{[1,\alpha]},\ h_j = g_j \mid_{[\alpha + 1,|g_j|]}$ for $\alpha = |c| = c(g_i^{-1},g_j)$, and so is $g_i \ast g_j$. Now, using (ER2) twice we can effectively compute $c((g_i \ast g_j)^{-1}, g_k)$, where $g_k = y_k^{\delta_k},\ \delta_k = \pm 1$. Indeed, $c((g_i \ast g_j)^{-1}, g_k) = c(h_j^{-1} \circ h_i^{-1}, g_k)$, so, if $c(h_j^{-1}, g_k) < |h_j^{-1}|$, then $c((g_i \ast g_j)^{-1}, g_k) = c(h_j^{-1}, g_k)$ which is computable by (ER2), and if $c(h_j^{-1}, g_k) \geqslant |h_j^{-1}|$, then $c((g_i \ast g_j)^{-1}, g_k) = |h_j| + c(h_i^{-1}, h_k)$, where $h_k = g_k \mid_{[|h_j|+1, |g_k|]}$ -- again, all components are computable and so is $c((g_i \ast g_j)^{-1}, g_k)$. 

From the above it follows that $y_i^{\pm 1} \ast y_j^{\pm 1} \ast y_k^{\pm 1}$ is a computable function for every $i, j, k \in [1, m]$. Continuing in the same way by induction one can show that every finite product of elements from $Y^{\pm 1}$, that is, every element of $G$ given as a finite product of generators and their inverses, is computable as a function defined over a computable segment in $\Lambda$ to $X^\pm$. Moreover, for any $g, h \in G$ one can effectively find $com(g,h)$ as a computable function. In particular, we automatically get a solution to the Word Problem in $G$.

\subsection{Effective hierarchy for $\Z^n$-free groups}
\label{subs:effect_h}

Now consider the case when $\Lambda = \Z^n$, where $\Z^n$ has the right lexicographic order. Recall that if $A$ and $B$ are ordered abelian groups, then the {\em right lexicographic order} on the direct sum $A \oplus B$ is defined as follows:
$$(a_1,b_1) < (a_2,b_2) \Leftrightarrow b_1 < b_2 \ \mbox{or} \ b_1 = b_2 \ \mbox{and} \ a_1 < a_2.$$
One can easily extend this definition to any number of components in the direct sum and apply it in the case of $\Z^n$ which is the direct sum of $n$ copies of $\Z$.

Every element $a \in \Z^n$ can be represented by an $n$-tuple $(a_1, \ldots, a_k, 0, \ldots, 0)$. We say that the {\em height} of $a$ is equal to $k$, and write $ht(a) = k$, if $a = (a_1, \ldots,$ $a_k, 0, \ldots, 0)$ and $a_k \neq 0$. If a group $G$ acts on a $\Z^n$-tree, then there is a Lyndon length function $l : G \to \Z^n$ and by the height $ht(g)$ of $g \in G$ we simply mean the height of its length $ht(l(g))$.

\smallskip

Consider a finitely generated $\Z^n$-free group $G$, where $n \in \N$. Using Bass-Serre theory one can represent $G$ as the fundamental group of a finite graph of groups with $\Z^{n-1}$-free vertex groups and maximal abelian (in the corresponding vertex groups) edge groups. Continuing this process inductively, one can obtain a finite hierarchy $\mathcal{G}$ of $\Z^k$-free groups, where $k < n$, such that $G$ can be built from groups in $\mathcal{G}$ by amalgamated free products and HNN-extensions along maximal abelian subgroups (see \cite{Serre:1980}, \cite{Bass:1991}). At the same time, by Theorem \ref{chis}, $G$ can be embedded into $CDR(\Z^n, X)$ for some alphabet $X$. Unfortunately, even if we know that the representation of $G$ by $\Z^n$-words over $X$ is effective, it does not give us effective representations of groups from the hierarchy $\mathcal{G}$ over $X$ and it is hard to use inductive arguments (if possible at all). Hence, in the case when $\Lambda = \Z^n$ we are going to introduce a stronger version of effective representation, which takes into account the hierarchical structure of $\Z^n$-free groups.

Suppose $n > 1$ and consider the $\Z^n$-tree $(\Gamma_G, d)$, which arises from the embedding of $G$ into $CDR(\Z^n, X)$ (see Section \ref{subs:universal} for details).

\smallskip

We say that $p, q \in \Gamma_G$ are {\em $\Z^{n-1}$-equivalent} ($p \sim q$) if $d(p, q) \in \Z^{n-1}$, that is, $d(p, q) = (a_1, \ldots, a_n),\ a_n = 0$. From metric axioms it follows that ``$\sim$'' is an equivalence relation and every equivalence class defines a {\em $\Z^{n-1}$-subtree} of $\Gamma_G$. Denote by $T_0$ the $\Z^{n-1}$-subtree of $\Gamma_G$ containing $\varepsilon$.

\smallskip

Let $\Delta_G = \Gamma_G / \sim$ and $\rho : \Gamma_G \rightarrow \Delta_G$ be the projection mapping. It is easy to see that $\Delta_G$ is a simplicial tree. Indeed, define $\widetilde{d} : \Delta_G \rightarrow \Z$ as follows:
\begin{equation}
\label{eq:new_length}
\forall\ \widetilde{p},\ \widetilde{q} \in \Delta_G:\ \ \widetilde{d}(\widetilde{p},\widetilde{q}) = k\ \ {\rm iff}\ \ d(p, q) = (a_1, \ldots, a_n)\ \ {\rm and}\ \ a_n = k.
\end{equation}
From metric properties of $d$ it follows that $\widetilde{d}$ is a well-defined metric on $\Delta_G$.

Since $G$ acts on $\Gamma_G$ by isometries, $p \sim q$ implies $g \cdot p \sim g \cdot q$ for every $g \in G$. Moreover, if $d(p, q) = (a_1, \ldots, a_n)$, then $d(g \cdot p, g \cdot q) = (a_1, \ldots, a_n)$, hence, $\widetilde{d}(g \cdot \widetilde{p}, g \cdot \widetilde{q}) = \widetilde{d}(\widetilde{p}, \widetilde{q})$. That is, $G$ acts on $\Delta_G$ by isometries, but the action is not free in general. From Bass-Serre theory it follows that $\Psi_G = \Delta_G / G$ is a graph in which vertices and edges correspond to $G$-orbits of vertices and edges of $\Delta_G$.

\begin{lemma}
\label{le:psi}
$\Psi_G$ is a finite graph.
\end{lemma}
\begin{proof} 
Let $G = \langle g_1, \ldots, g_k \rangle$. Let $K$ be a subtree of $\Gamma_G$ spanned by $g_i^{\pm 1} \cdot \varepsilon,\ i \in [1, k]$. It is easy to see that if $|g_i| = (a_{i_1}, \ldots, a_{i_n}) \in \Z^n$, where $a_{i_n} \geqslant 0,\ i \in [1, k]$ then $\widetilde{K} = \rho(K) \subset \Delta_G$ is a finite subtree such that
$$|V(\widetilde{K})| \leqslant 2 \sum_{i=1}^n a_{i_n},$$
where $V(\widetilde{K})$ denotes the set of vertices in $\widetilde{K}$.

Now we claim that for every $q \in \Gamma_G$ there exists $p \in K$ and $g \in G$ such that $q = g \cdot p$. Indeed, since $\Gamma_G$ is spanned by $g \cdot \varepsilon,\ g \in G$, let $h \in G$ be such that $p = \langle \alpha, h \rangle,\ h = h_1 \cdots h_m$, where $h_j \in \{g_1^{\pm 1}, g_2^{\pm 1}, \ldots, g_k^{\pm 1}\}$. Then
$$[\varepsilon, h \cdot \varepsilon] \subset [\varepsilon,\ h_m] \cup [h_{m-1} \cdot \varepsilon,\ (h_{m-1} h_m) \cdot \varepsilon] \cup \cdots$$
$$ \cup [(h_1 \cdots h_{m-1}) \cdot \varepsilon,\ (h_1 \cdots h_m) \cdot \varepsilon].$$
It follows that $q \in [(h_j \cdots h_{m-1}) \cdot \varepsilon,\ (h_j \cdots h_m) \cdot \varepsilon]$ for some $j$ and $(h_j \cdots h_{m-1})^{-1} \cdot q = p \in [\varepsilon,\ h_m \cdot \varepsilon] \subset K$. So $q = g \cdot p$ for $g = h_j \cdots h_{m-1}$ as required.

From the claim it follows that $\Gamma_G$ is spanned by translates of $K$, so $\Delta_G$ is spanned by translates of $\widetilde{K}$. Hence, there can be only finitely many $G$-orbits of vertices and edges in $\Delta_G$, and $\Psi_G$ is a finite graph.
\end{proof}

%So, let $|V(\Psi_G)| = m$ and $\mathcal{T}_1, \ldots, \mathcal{T}_m$ be these $G$-orbits.

From Lemma \ref{le:psi} it follows that the number of $G$-orbits of $\Z^{n-1}$-subtrees in $\Gamma_G$ is finite and it is equal to the number of vertices in $\Psi_G$.

Consider the graph $\Psi_G$ more closely. The set of vertices and edges of $\Psi_G$ we denote respectively by $V(\Psi_G)$ and $E(\Psi_G)$ so that the functions
$$\sigma : E(\Psi_G) \rightarrow V(\Psi_G), \ \ \tau : E(\Psi_G) \rightarrow V(\Psi_G),\ \ ^- : E(\Psi_G) \rightarrow E(\Psi_G)$$
of taking the initial vertex, terminal vertex, and inverting an edge satisfy the following conditions:
$$\sigma(\bar{e}) = \tau(e),\ \tau(\bar{e}) = \sigma(e),\ \bar{\bar{e}} = e,\ \bar{e} \neq e.$$
Let $\mathcal{T}$ be a maximal subtree of $\Psi_G$ and let $\pi : \Delta_G \rightarrow \Delta_G / G = \Psi_G$ be the canonical projection of $\Delta_G$ onto its quotient, so that $\pi(v) = G v$ and $\pi(e) = G e$ for every $v \in V(\Delta_G),\ e \in E(\Delta_G)$. There exists an injective morphism of graphs $\eta : \mathcal{T} \rightarrow \Delta_G$ such that $\pi \circ \eta = id_\mathcal{T}$ (see Section 8.4 of \cite{Cohen}), in particular, $\eta(\mathcal{T})$ is a subtree of $\Delta_G$. One can extend $\eta$ to a map (which we denote by $\eta$ again) $\eta: \Psi_G \rightarrow \Delta_G$ so that $\eta$ maps vertices to vertices, edges to edges, and $\pi \circ \eta = id_{\Psi_G}$ as follows. Choose an orientation $O$ of the graph $\Psi_G$ and let $e \in O - \mathcal{T}$. There exists an edge $e^\prime \in \Delta_G$ such that $\pi(e^\prime) = e$. Clearly, $\sigma(e^\prime)$ and $\eta(\sigma(e))$ are in the same $G$-orbit. Hence, $g \cdot \sigma(e^\prime) = \eta(\sigma(e))$ for some $g \in G$. Define $\eta(e) = g \cdot e^\prime$ and set $\eta(\bar{e}) = \overline{\eta(e)}$. Note that the vertices $\eta(\tau(e))$ and $\tau(\eta(e))$ are in the same $G$-orbit. Hence, there exists an element $\gamma_e \in G$ such that $\gamma_e \cdot \tau(\eta(e)) = \eta(\tau(e))$.

Put
$$G_v = Stab_G(\eta(v)),\ G_e = Stab_G(\eta(e))$$
and define boundary monomorphisms as inclusion maps $i_e : G_e \hookrightarrow G_{\sigma(e)}$ for edges $e \in \mathcal{T} \cup O$ and as conjugations by $\gamma_{\bar{e}}$ for edges $e \notin \mathcal{T} \cup O$, that is,
\[ i_e(g) = \left\{ \begin{array}{ll}
\mbox{$g$,} & \mbox{if $e \in \mathcal{T} \cup O$,} \\
\mbox{$\gamma_{\bar{e}} g \gamma_{\bar{e}}^{-1}$,} & \mbox{if $e \notin \mathcal{T} \cup O$.}
\end{array}
\right. \]
According to the Bass-Serre structure theorem we have
\begin{equation}
\label{eq:presentation}
G \simeq \pi(\mathcal{G},\Psi_G,\mathcal{T}) = \langle G_v \ (v \in V(\Psi_G)),\ \gamma_e\ (e \in E(\Psi_G)) \mid rel(G_v),
\end{equation}
$$\gamma_e i_e(g) \gamma_e^{-1} = i_{\bar{e}}(g)\ (g \in G_e),\ \gamma_e \gamma_{\bar{e}} = 1,\ \gamma_e = 1\ (e \in \mathcal{T})\rangle.$$

\smallskip

Let $\mathcal{K} = \rho^{-1}( \eta( \mathcal{T} ) ),\ \overline{\mathcal{K}} = \rho^{-1}(\eta(\Psi_G))$, hence, $\mathcal{K},\ \overline{\mathcal{K}}$ are subtrees of $\Gamma_G$ such that $\mathcal{K} \subseteq \overline{\mathcal{K}}$. Obviously $T_0 \subseteq \mathcal{K}$. Moreover, both $\mathcal{K}$ and $\overline{\mathcal{K}}$ contain finitely many $\Z^{n-1}$-subtrees, and meet every $G$-orbit of $\Z^{n-1}$-subtrees of $\Gamma_G$.

Every $v \in V(\Psi_G)$ lifts to a $\Z^{n-1}$-subtree $T_{\eta(v)}$ of $\Gamma_G$, where $\eta(v) = \rho(T_{\eta(v)})$. Clearly $Stab_G(\eta(v)) = Stab_G(T_{\eta(v)})$. 

Recall that $T_0$ is the $\Z^{n-1}$-subtree of $\Gamma_G$ containing $\varepsilon$. Hence, $T_0 \subset \mathcal{K}$ and $Stab_G(T_0)$ is a subgroup of $CDR(\Z^{n-1}, X)$. The stabilizer of any other $\Z^{n-1}$-subtree of $\mathcal{K}$ is conjugate to a subgroup of $CDR(\Z^{n-1}, X)$, as shown in the next lemma below. Before we prove it, recall (see \cite[Section 3.1]{Chiswell:2001}, for example) that the {\em axis} of an element $h$ of a group $H$ acting on a $\Lambda$-tree $\Gamma$ is the subset $Axis(h)$ of $\Gamma$ defined as follows:
$$Axis(h) = \{ p \in \Gamma \mid [h^{-1} p, p] \cap [p, h p] = \{p\} \}.$$
If $H$ is a $\Lambda$-free group, then every element acts hyperbolically on $\Gamma$ and $Axis(h)$ is a closed non-empty $\langle h \rangle$-invariant subtree of $\Gamma$. 

\begin{lemma}
\label{le:stab}
Let $T$ be a $\Z^{n-1}$-subtree of $\mathcal{K}$. Then
$$Stab_G(T) = f_T \ast K_T \ast f_T^{-1},$$
where $K_T$ is a subgroup of $CDR(\Z^{n-1}, X)$ (possibly trivial) and $f_T = \mu([\varepsilon, x_T])$ $\in CDR(\Z^n, X)$ for some point $x_T \in T$. Moreover, if $Stab_G(T)$ is not trivial, then $x_T \in Axis(g) \cap T$ for some $g \in Stab_G(T)$.
\end{lemma}
\begin{proof} 
If $Stab_G(T)$ is trivial then the statement obviously holds.

\smallskip

Suppose $Stab_G(T) \neq 1$ and let $g \in Stab_G(T)$. By Corollary 1.6 \cite{Chiswell:2005}, $Axis(g)$ meets every $\langle g \rangle$-invariant subtree of $\Gamma_G$. Since $T$ is $\langle g \rangle$-invariant, we have $Axis(g) \cap T \neq \emptyset$. Hence, choose some $x_T \in Axis(g) \cap T$ and put $f_T = \mu([\varepsilon, x_T])$. We have $g \cdot x_T \in T$, so $|f_T^{-1} \ast g \ast f_T| \in \Z^{n-1}$, in other words, $g = f_T \ast a_g \ast f_T^{-1},\ a_g \in CDR(\Z^{n-1}, X)$. Since $Stab_G(T)$ is a group, we have $K_T = \{ a_g \mid g \in Stab_G(T)\}$ is a subgroup of $CDR(\Z^{n-1}, X)$.
\end{proof}

Let $e$ be an edge of $\Psi_G$ such that $e \in O,\ e \notin \mathcal{T}$. Let $v = \sigma(\eta(e)) = \eta(\sigma(e)),\ w =\tau(\eta(e))$ and $u = \eta(\tau(e)) = \gamma_e \cdot w$. We have $u, v \in \eta(\mathcal{T}),\ w \notin \eta(\mathcal{T})$. Hence,
$$\gamma_e\ Stab_G(w)\ \gamma_e^{-1} = Stab_G(u).$$
By definition we have $i_e(G_e) \subseteq G_v = Stab_G(T)$, where $T = \rho^{-1}(v)$ and $i_{\bar{e}}(G_e) = \gamma_e G_e \gamma_e^{-1} \subseteq G_u = Stab_G(S)$, where $S = \rho^{-1}(u)$. Thus, we have $i_e(G_e) = f_T \ast A \ast f_T^{-1},\ i_{\bar{e}}(G_e) = f_S \ast B \ast f_S^{-1}$, where $A \leqslant K_T$ and $B \leqslant K_S$ are isomorphic abelian subgroups of $CDR(\Z^{n-1}, X)$. So,
$$\gamma_e \ast (f_T \ast A \ast f_T^{-1}) \ast \gamma_e^{-1} = f_S \ast B \ast f_S^{-1}$$
and it follows that $f_S^{-1} \ast \gamma_e \ast f_T = r_e \in CDR(\Z^n, X)$ so that $r_e \ast A \ast r_e^{-1} = B$. Thus, we have
$$\gamma_e = f_S \ast r_e \ast f_T^{-1}.$$
Observe that $r_e \in CDR(\Z^n, X) - CDR(\Z^{n-1}, X)$ because otherwise $\gamma_e \cdot T = S$, that is, $u = v,\ S = T$ and thus $\gamma_e \in Stab_G(T)$ - a contradiction.

\smallskip

Now, we can give an inductive definition of effective hierarchy for a finitely generated $\Z^n$-free group $G$. We say that $G$ has an {\em effective hierarchy over an alphabet $X$} if the following conditions are satisfied.
\begin{enumerate}
\item[(EFH1)] If $n = 1$, then $G$ has an effective representation by $\Z$-words over the alphabet $X$.
\item[(EFH2)] If $n > 1$, then in the presentation (\ref{eq:presentation}) for $G$ 
\begin{enumerate}
\item[(a)] each vertex group $G_v$ is given in the form $f_v \ast K_v \ast f_v^{-1}$, where $K_v$ has an effective hierarchy over $X$ (we assume that effective hierarchy is defined for $K_v$ by induction) and $f_v$ is a computable $\Z^n$-word over $X$,
\item[(b)] for each edge group $G_e$, its images in the corresponding vertex groups have effective representations over $X$,
\item[(c)] each $r_e$ in $\gamma_e$ represented as the product $\gamma_e = f_S \ast r_e \ast f_T^{-1}$ is given as a
computable $\Z^n$-word over $X$.
\end{enumerate}
\end{enumerate}

Observe that from the definition above it follows that effective hierarchy over $X$ implies effective representation over $X$.

\section{Effective regular completions}
\label{sec:reg_comp}

Let $G$ be a finitely generated subgroup of $CDR(\Z^n, X)$, where $\Z^n$ is ordered with respect to the right lexicographic order and $X$ is an arbitrary alphabet ($X$ can be infinite). We are going to construct a finitely generated subgroup $H$ of $CDR(\Z^n, Y)$, where $Y$ is a finite alphabet, such that the length function on $H$ induced from $CDR(\Z^n, Y)$ is regular and $G$ embeds into $H$. Moreover, the embedding preserves the length of elements of $G$. In other words, we are going to construct a finitely generated {\em $\Z^n$-completion} of $G$ (see \cite{KMRS:2012}). Finally, if $G$ has an effective hierarchy over $X$, then we show that the construction of $H$ is effective and it has an effective hierarchy over $Y$.

The argument is conducted by induction on $n$ as follows. $G \subset CDR(\Z^n, X)$ splits into the fundamental group of a finite graph of groups $\Psi_G$, where each vertex group is isomorphic to a subgroup of $CDR(\Z^{n-1}, X)$. Inductively we can assume that we can construct a regular completion for each vertex group and then we combine these regular completions to form a regular completion for $G$ itself. If the constructed regular completions for the vertex groups are effective, then the regular completion for $G$ is shown to be effective too.

\subsection{Simplicial case}
\label{subs:n=1}

Let $G$ be a finitely generated subgroup of $CDR(\Z, X)$. Hence, $\Gamma_G$ is a simplicial tree and $\Delta = \Gamma_G / G$ is a folded $X$-labeled digraph (see \cite{Kapovich_Myasnikov:2002}) with labeling induced from $\Gamma_G$. $\Delta$ is finite, which follows from the fact that $G$ is finitely generated and from the construction of $\Gamma_G$. Moreover, $\Delta$ recognizes $G$ with respect to some vertex $v$ (the image of $\varepsilon$) in the sense that $g \in CDR(\Z, X)$ belongs to $G$ if and only if there exists a loop in $\Delta$ at $v$ such that its label is exactly $g$.

The following lemma provides the required result.

\begin{lemma}
\label{le:n=1}
Let $G$ be a finitely generated subgroup of $CDR(\Z, X)$. Then there exists a finite alphabet $Y$ and an embedding $\phi : G \rightarrow H$, where $H = F(Y)$, inducing an embedding $\psi : \Gamma_G \rightarrow \Gamma_H$ such that
\begin{enumerate}
\item[(i)] $|g|_X = |\phi(g)|_Y$ for every $g \in G$,
\item[(ii)] if $A$ is a maximal abelian subgroup of $G$, then $\phi(A)$ is a maximal abelian subgroup of $H$,
\item[(iii)] if $a$ and $b$ are non-$G$-equivalent ends of $\Gamma_G$, then $\psi(a)$ and $\psi(b)$ are non-$H$-equivalent ends of $\Gamma_H$,
\item[(iv)] if $A$ and $B$ are maximal abelian subgroups of $G$, which are not conjugate in $G$, then $\phi(A)$ and $\phi(B)$ are not conjugate in $H$.
\end{enumerate}
Moreover, if $G$ has an effective representation by $\Z$-words over $X$, then $Y$ can be found effectively and the embedding $\phi : G \rightarrow H$ is effective.
\end{lemma}
\begin{proof} 
Since $G$ is finitely generated, there are only finitely many letters which are used in the representation of the generating set of $G$, so, $X$ can be assumed to be finite. Consider $\Delta = \Gamma_G / G$ and let $v \in V(\Delta)$ be the image of $\varepsilon$. Let $E = \{e_1,\ldots, e_k\}$ be the set of edges of $\Delta$ and $E_+$ an orientation on $E$. Take a copy of $\Delta$ denoted $\Delta'$, which has the set of edges $E'$ and orientation $E'_+$ corresponding to $E$ and $E_+$ in $\Delta$. Let $v' \in V(\Delta')$ correspond to $v \in V(\Delta)$. Introduce a labeling function $\mu'$ on edges of $\Delta'$ as follows: $\mu'(e_i) = e_i$ if $e_i \in E'_+$, and $\mu'(e_i) = e_i^{-1}$ if $e_i \in E' - E'_+$. Hence, $\mu' : E' \rightarrow E'$ and $\Delta'$ naturally becomes a $E'$-labeled digraph. There exists a natural isomorphism of graphs $\gamma : \Delta \rightarrow \Delta'$, which induces a natural isomorphism $\phi : G \rightarrow G'$, where $G' \leqslant F(E')$ is recognized by $\Delta'$ with respect to $v' \in V(\Delta')$. Let $Y = E'$ and $H
= F(Y)$. Since $G'$ is a subgroup of $F(Y)$, we obtain that $\Gamma_{G'}$ naturally embeds into $\Gamma_H$, which is the Cayley graph of $H$ with respect to $Y$. Now, $\phi : G \rightarrow G'$ induces an isomorphism between $\Gamma_G$ and $\Gamma_{G'}$, which gives an embedding $\psi : \Gamma_G \rightarrow \Gamma_H$.

\smallskip

Observe that under the assumption that $G$ has an effective representation by $\Z$-words over $X$, one can effectively enumerate $X$ and $\Delta$ can be constructed effectively from a finite wedge of loops labeled by finite words (which again can be found effectively) corresponding to the generators of $G$. Hence, $G'$ has an effective representation by $\Z$-words over a finite
alphabet $Y$ (that is found effectively) meaning that the embedding $G' \hookrightarrow H$ is effective. Hence, the embedding $G \hookrightarrow H$ is effective too. Now, we prove the required properties of this embedding.

\smallskip

First of all, from the construction of $\Delta'$ it follows that $|g|_X = |\phi(g)|_Y$ for every $g \in G$ and (i) follows.

\smallskip

Next, if $g \in G$ is not a proper power in $G$, then $\phi(g)$ is not a proper power in $F(Y)$. Indeed, if $\phi(g) \in G'$ is a proper power in $F(Y)$, then, due to one-to-one correspondence between edges of $\Delta'$ and their labels, there exists a reduced path $q = q_1 q_2 q_1^{-1} \in \Delta'$ at $v'$ such that $\mu'(q_2)$ is cyclically reduced, $q_2 = q_3^m,\ m > 1$ for some loop $q_3$ at $t(q_1)$, and $\mu'(q) = \phi(g)$. Hence, $\phi(g)$ is a proper power in $G'$ and $g = \mu(\gamma^{-1}(q))$ is a proper power in $G$. So, (ii) follows.

\smallskip

Now we prove (iii). Let $a$ and $b$ be two ends of $\Gamma_G$. Note that $a$ and $b$ correspond to unique infinite geodesic rays $r_a$ and $r_b$ in $\Gamma_G$ originating at $\varepsilon$, whose edges are labeled by $X^{\pm}$. Next, $r_a$ and $r_b$ correspond to reduced infinite paths $p_a$ and $p_b$ in $\Delta$ starting at $v$. Now, consider the images of $a$ and $b$ under $\psi$: these are the ends $\psi(a)$ and $\psi(b)$ of $\psi(\Gamma_G)$, hence, of $\Gamma_H$. They correspond to paths $p_{\psi(a)}$ and $p_{\psi(b)}$, which are the images of $p_a$ and $p_b$ under $\gamma : \Delta \rightarrow \Delta'$. Both $p_{\psi(a)}$ and $p_{\psi(b)}$ are reduced infinite paths in $\Delta'$ starting at $v'$. If $\psi(a)$ and $\psi(b)$ are $H$-equivalent then $\mu'(p_{\psi(b)}) = h \mu'(p_{\psi(a)})$, where $h \in H$. But since there exists one-to-one correspondence between edges of $\Delta'$ and their labels, it follows that $h = \mu'(p')$, where $p'$ is a loop at $v'$, and $p_{\psi(b)} = p' p_{\psi(a)}$. In other words, $h \in G'$ and the ends $\psi(a)$ and $\psi(b)$ are $G'$-equivalent. Finally, the loop $p'$ can be lifted to the loop $p = \gamma^{-1}(p')$ at $v$ in $\Delta$ 
such that $g = \mu(p) \in G$, which implies that $a$ and $b$ are $G$-equivalent. So, (iii) follows.

\smallskip

In order to prove (iv) observe that a maximal abelian subgroup $C$ of $G$ corresponds to a pair of ends of $\Gamma_G$, which are the ends of its axis $Axis(C)$. Now, (iv) follows from (iii).
\end{proof}

Lemma \ref{le:n=1} can be generalized as follows.

\begin{corollary}
\label{co:n=1}
Let $G$ be a finitely generated subgroup of $CDR(\Z, X)$. Assume that $\Gamma_G$ is embedded into a $\Z$-tree $T$, whose edges are labeled by $X^{\pm}$ so that the action of $G$ on $\Gamma_G$ extends to an action of $G$ on $T$, and there are only finitely many $G$-orbits of ends of $T$, which belong to $T - \Gamma_G$. Then there exists a finite alphabet $Y$, a $\Z$-tree $T'$, whose edges are labeled by $Y^{\pm}$, and a finitely generated subgroup $H \subseteq CDR(\Z, Y)$ such that $\Gamma_H$ is embedded into $T'$ so that the action of $H$ on $\Gamma_H$ extends to a regular action of $H$ on $T'$. Also, there is an embedding $\theta: T \to T'$, where $\theta(\Gamma_G) \subseteq \Gamma_H$, which indices an embedding $\phi : G \rightarrow H$ such that
\begin{enumerate}
\item[(i)] $|g|_X = |\phi(g)|_Y$ for every $g \in G$,
\item[(ii)] if $A$ is a maximal abelian subgroup of $G$, then $\phi(A)$ is a maximal abelian subgroup of $H$,
\item[(iii)] if $a$ and $b$ are non-$G$-equivalent ends of $T$, then $\theta(a)$ and $\theta(b)$ are non-$H$-equivalent ends of $T'$.
\end{enumerate}
Moreover, if
\begin{enumerate}
\item[(e1)] $G$ has an effective representation by $\Z$-words over $X$, and
\item[(e2)] a set of representatives $q_1, \ldots, q_m$ of $G$-orbits of ends of $T$, which belong to $T - \Gamma_G$ is given as a set of functions $q_i : [1, \infty) \to X^\pm$ so that each $q_i$ is computable,
\end{enumerate}
then $Y$ can be found effectively, $H$ has an effective representation by $\Z$-words over $Y$, and the embedding $\phi : G \rightarrow H$ is effective.
\end{corollary}
\begin{proof} 
Consider $T / G$. Since there are only finitely many $G$-orbits of ends in $T - \Gamma_G$, we have that $T / G$ consists of $\Delta = \Gamma_G / G$ and a forest formed by finitely many infinite rays attached to some vertices of $\Delta$. These rays correspond to $G$-orbits of ends in $T - \Gamma_G$. By Lemma \ref{le:n=1} there exists a finite alphabet $Y_1$ and a relabeling of edges of $\Delta$ by $Y_1^{\pm}$, which induces embeddings $\psi: \Gamma_G \to \Gamma_{F(Y_1)}$ and $\phi : G \rightarrow F(Y_1)$ in such a way that
\begin{enumerate}
\item[(a)] $|g|_X = |\phi(g)|_{Y_1}$ for every $g \in G$,
\item[(b)] if $A$ is a maximal abelian subgroup of $G$, then $\phi(A)$ is a maximal abelian subgroup of $F(Y_1)$,
\item[(c)] if $a$ and $b$ are non-$G$-equivalent ends of $\Gamma_G$, then $\psi(a)$ and $\psi(b)$ are non-$F(Y)$-equivalent ends of $\Gamma_{F(Y_1)}$.
\end{enumerate}
Now, the new labeling of $\Delta$ can be extended to a labeling of $T / G$ as follows. Since there are only finitely many infinite rays in $T / G$, then there are only finitely many branch-points in $T / G - \Delta$. Hence, in $T / G - \Delta$ there are finitely many paths $p_{v, w}$ connecting either two adjacent branch-points $v$ and $w$ in $T / G - \Delta$, or a vertex $v$ of $\Delta$ , that an infinite ray from $T / G - \Delta$ is attached to, and a branch-point $w$ in $T / G - \Delta$. Denote the set of all such paths $p_{v, w}$ by $P$. Next, there are finitely many infinite rays $R = \{r_1,\ldots, r_m\}$ in $T / G - \Delta$, which do not contain any branch-points. Eventually, every infinite ray $r$ in $T / G - \Delta$ corresponding to a $G$-orbit of ends in $T - \Gamma_G$, attached to a vertex $u_a$ in $\Delta$ can be decomposed as the concatenation of paths
$$p_{u_a, v_1}\ p_{v_1, v_2}\ \cdots\ p_{v_{k-1}, v_k}\ r_i,$$
where $p_{u_a, v_1}, \ldots, p_{v_{k-1}, v_k} \in P$ and $r_i \in R$. Now for each $p \in P \cup R$, relabel all edges of $p$ by the same letter $a_p$ so that $a_p \neq a_q$ whenever $p \neq q,\ p,q \in P \cup R$, and let $Y_2 = \{a_p \mid p \in P \cup R\}$. Observe that $Y_2$ is finite even if $X$ is not finite. Let $W = \{ \mu(p) \mid p \in P \}$. Hence, the label of any infinite ray in $T / G - \Delta$ can be decomposed as
$$w_1\ w_2\ \cdots\ w_k\ \mu(r)$$
for some $w_i \in P,\ r \in R$.

\smallskip

Let $Y = Y_1 \cup Y_2$ and $H = \langle Y_1 \cup W \rangle \subset CDR(\Z, Y)$. Let $Q$ be the collection of all paths of the form $p r$, where $\mu(p) \in H$, and either $r \in R$, or $r$ is empty. Define $T' = Q / \sim$, where ``$\sim$'' stands for identification of common initial subwords for every pair $p_1, p_2 \in Q$. Hence, $T'$ is a $\Z$-tree labeled by $Y$, that contains $\Gamma_H$ as a subtree, on which $H$ acts regularly by left multiplication. $T'$ contains a copy of $T$, relabeled as shown above, which provides an embedding $\theta: T \rightarrow T'$, where $\theta(\Gamma_G) \subseteq \Gamma_H$. Finally, non-$G$-equivalent ends of $T$ are labeled by different letters in $\theta(T)$, so combined with (c) it implies (iii).

Assuming (e1), it follows that $Y_1$ can be found effectively by Lemma \ref{le:n=1}. Next, using (e2) we can find all branch-points in $T / G - \Delta$ effectively. Indeed, each $q_i$ can be effectively represented as the reduced concatenation $g_i \circ q'_i$, where $g_i \in F(X)$ and $q'_i$ is an infinite ray in $T / G - \Delta$ ($g_i$ can be found by ``reading'' the label of $q_i$ in $\Delta$ letter by letter, the process stops because $q_i$ cannot be read in $\Delta$ completely starting from $v$). Thus, one can find out effectively if $q'_i$ and $q'_j$ originate from the same vertex in $\Delta$ and, in the case they do, determine their maximal common initial segment, which gives a branch-point (since $q_i$ and $q_j$ are not $G$-equivalent for $i \neq j$, the maximal common initial segment of $q'_i$ and $q'_j$ is finite). The number of branch-points is bounded by the number of orbits of infinite rays in $T / G$, so eventually one can find all of them. Hence, $Y_2$ and $W$ can be found effectively, so $H$ has an effective representation by $\Z$-words over $Y$ and effectiveness of the embedding of $G$ into $H$ follows from the effectiveness of $\phi : G \rightarrow F(Y_1)$.
\end{proof}

\subsection{General case}
\label{subs:general}

Let $G$ be a finitely generated subgroup of $CDR(\Z^n, X)$ for some alphabet $X$. We are going to use the notations introduced in Subsection \ref{subs:effect_h}, that is, we assume that $\mathcal{K},\ \Psi_G,\ \Delta_G$ etc. are defined for $G$ as well as the presentation (\ref{eq:presentation}).

\smallskip

Our first step is to relabel edges of $\Gamma_G$ so that non-$G$-equivalent $\Z^{n-1}$-subtrees are labeled by disjoint alphabets.

\smallskip

Recall that every edge $e$ in $\Gamma_G$ is labeled by a letter $\mu(e) \in X^\pm$. Let $T$ be a $\Z^{n-1}$-subtree of $\mathcal{K}$ and $X_T$ a copy of $X$ (disjoint from $X$) so that we have a bijection $\pi_T : X \rightarrow X_T$, where $\pi_T(x^{-1}) = \pi_T(x)^{-1}$ for every $x \in X$. We assume $X_S \cap X_T = \emptyset$ for distinct $S, T \in \mathcal{K}$. Let $\Gamma'$ be a copy of $\Gamma_G$ and $\nu : \Gamma' \rightarrow \Gamma_G$ a natural bijection (the bijection on points naturally induces the bijection on edges). Denote $\varepsilon' = \nu^{-1}(\varepsilon)$.

Let $X' = \bigcup \{ X_T \mid T \in \mathcal{K} \}$. We introduce a labeling function $\mu' : E(\Gamma') \rightarrow X'^{\pm}$ as follows: $\mu'(e) = \pi_T(\mu(\nu(e)))$ if $\nu(e) \in T$. $\mu'$ naturally extends to the labeling of paths in $\Gamma'$. 

Recall that $V_G = \{ v \in \Gamma_G \mid\ \exists\ g \in G:\ v = \langle |g|,g \rangle \}$, that is, $V_G$ is the collection of points of $\Gamma_G$ that are in one-to-one correspondence with the set of all elements of $G$. Now, if $V' = \nu^{-1}(V_G)$, then define
$$G' = \{ \mu'(p) \mid p = [\varepsilon', v']\ {\rm for\ some}\ v' \in V'\}.$$

\begin{lemma}
\label{le:relabel}
$G'$ is a subgroup of $CDR(\Z^n, X')$, which acts freely on $\Gamma'$ and there exists an isomorphism $\phi : G \rightarrow G'$ preserving the length Lyndon function, that is, $L_\varepsilon(g) = L_{\varepsilon'}(\phi(g))$. Moreover, if $G$ has an effective hierarchy over $X$, then $G'$ has an effective hierarchy over $X'$.
\end{lemma}
\begin{proof} 
Take $g \in G$. Since $g = \mu([\varepsilon, v]) \in CDR(\Z^n, X)$ for some $v \in V_G \subset \Gamma_G$, then define
$$\phi(g) = \mu'( [\varepsilon', v'] ) \in G',$$
where $v' = \nu^{-1}(v)$. All the required properties of $G'$ follow immediately.

The effectiveness part is obvious.
\end{proof}

According to Lemma \ref{le:relabel} we have $\Gamma' = \Gamma_{G'}$. Observe that the structure of $\Z^{n-1}$-trees in $\Gamma_{G'}$ is the same as in $\Gamma_G$. Hence, if ``$\sim$'' is a $\Z^{n-1}$-equivalence of points of $\Gamma_{G'}$ then $\Delta_{G'} = \Gamma_{G'} / \sim$ and $\Psi_{G'} = \Delta_{G'} / G'$ are naturally isomorphic respectively to $\Delta_G = \Gamma_G / \sim$ and $\Psi_G = \Delta_G / G$. So, with a slight abuse of notation let $X = X',\ G = G'$.

\smallskip

The next step is to refine the labeling so as to make the alphabet $X$ finite. To do this we have to analyze the structure of the $\Z^{n-1}$-subtrees of $\mathcal{K}$. Recall that if $x, y, z$ are points in a $\Lambda$-tree, then the intersection of the geodesics $[x, y] \cap [x, z]$ equals the geodesic $[x, w]$ for some point $w$ in the tree. This point is unique and we define $Y(x, y, z) = w$. Note that $Y(x, y, z)$ does not depend on the order of points in the triple $\{x, y, z\}$.

\begin{lemma}
\label{le:trivial_stab}
Let $T$ be a $\Z^{n-1}$-subtree of $\mathcal{K}$ such that $Stab_G(T)$ is trivial. Then $T$ contains only finitely many branch-points and each branch-point of $T$ is of the form $Y(\varepsilon, x, y)$, where $x, y \in \{x_S\ (S \in \mathcal{K}),\ \gamma_e^{\pm 1} \cdot \varepsilon\ (e \in \Psi_G)\}$. Moreover, if $G$ has an effective hierarchy over $X$, then all branch-points of $T$ can be found effectively.
\end{lemma}
\begin{proof} 
Suppose $a$ is a branch-point of $T$. Then there exist $x_1, x_2 \in \Gamma_G - T$ such that $a = Y(\varepsilon, x_1, x_2)$. Indeed, otherwise, from the construction of $\Gamma_G$ it follows that there exist distinct $g, h \in G$ such that $g \cdot \varepsilon,\ h \cdot \varepsilon \in T$, so $g^{-1} h \in Stab_G(T)$, which is a contradiction. Without loss of generality we can assume that $x_1 \in S_1,\ x_2 \in S_2$, where $S_1$ and $S_2$ are $\Z^{n-1}$-subtrees of $\Gamma_G$ adjacent to $T$. Observe that $S_1$ and $S_2$ belong to distinct $G$-orbits because $Stab_G(T)$ is trivial. Thus, the number of branch-points in $T$ is finite. Finally, the pair $(T, S_i),\ i = 1,2$ corresponds to an edge $e_i = (\pi(\rho(T)), \pi(\rho(S_i)))$ of $\Psi_G$. If $e_i \in \mathcal{T}$ then $x_i$ can be chosen to be $x_{S_i}$, otherwise  $x_i$ can be chosen to be $\gamma_e \cdot \varepsilon$.

The effectiveness part of the statement follows immediately.
\end{proof}

In particular, from Lemma \ref{le:trivial_stab} it follows that every $\Z^{n-1}$-subtree $T$ of $\mathcal{K}$ with a trivial stabilizer can be relabeled by a finite alphabet. Indeed, $T$ may be cut at its branch-points into finitely many closed segments and half-open rays that do not contain any branch-points. Then all these segments and rays can be labeled by different letters (all points in each piece is labeled by one letter).

In the case of non-trivial stabilizer the situation is a little more complicated.

\begin{lemma}
\label{le:stab1}
Let $T$ be a $\Z^{n-1}$-subtree of $\mathcal{K}$ such that $Stab_G(T) = f_T \ast K_T \ast f_T^{-1}$, where $K_T \subset CDR(\Z^{n-1}, X)$ is non-trivial. Then $\Gamma_{K_T}$ embeds into $T$ (the base-point of $\Gamma_{K_T}$ is identified with $x_T$), the action of $K_T$ on $\Gamma_{K_T}$ extends to the action of $K_T$ on $T$ and the following statements hold.
\begin{enumerate}
\item[(a)] Every end of $T - \Gamma_{K_T}$ is $K_T$-equivalent to one of the ends of a subtree with finitely many ends, which is the intersection of $T$ with the union of the segments $[\varepsilon, x_S],\ S \in \mathcal{K}$.
\item[(b)] Every end $a$ of $T - \Gamma_{K_T}$ extends the axis of some (possibly trivial) centralizer $C_a$ of $K_T$.
\item[(c)] There are only finitely many $K_T$-orbits of branch-points of $T - \Gamma_{K_T}$.
\item[(d)] If $K_T \subset CDR(\Z^{n-1}, Y)$ for some finite alphabet $Y$, then the labeling of $\Gamma_{K_T}$ by $Y$ can be $K_T$-equivariantly extended to a labeling of $T$ by a finite extension $Y'$ of $Y$.
\end{enumerate}
Moreover, if $G$ has an effective hierarchy over $X$, then
\begin{itemize}
\item the centralizer $C_a$ in (b) can be found effectively, 
\item representatives of $K_T$-orbits of branch-points of $T - \Gamma_{K_T}$ in (c) can be found effectively, and 
\item the new alphabet $Y'$ in (d) can be found effectively provided $Y$ is given.
\end{itemize}
\end{lemma}
\begin{proof}
The statements that $\Gamma_{K_T}$ embeds into $T$ and the action of $K_T$ on $\Gamma_{K_T}$ extends to the action of $K_T$ on $T$ follow from the definition of $K_T$. Then, observe that (a) follows immediately from the structure of $\Gamma_G$ explained in detail in Subsection \ref{subs:effect_h}.

\smallskip

Let us prove (b). From (a) it follows that there exist only finitely many $K_T$-orbits of ends of $T - \Gamma_{K_T}$. Moreover, in every such orbit one can choose a representative $a$, which is an end of the intersection of $T$ with the union of the segments $[\varepsilon, x_S],\ S \in \mathcal{K}$. Each such $a$ is associated with an edge  $e = (\pi(\rho(T)), \pi(\rho(S)))$ of $\Psi_G$, where $S \in \mathcal{K}$: $e$ is essentially a pair of ends of $T$ and $S$ of full $\Z^{n-1}$-type, and the end of $T$ is exactly $a$. Next, $e$ is associated with two maximal abelian subgroups of $Stab_G(T)$ and $Stab_G(S)$. If $E_a < Stab_G(T)$ is such a maximal abelian subgroup, then $C_a = f_T^{-1} \ast E_a \ast f_T$ is a maximal abelian subgroup of $K_T$. Note that since $a$ is not an end of $\Gamma_{K_T}$, we have that $ht(C_a) < n - 1$ (by which we mean the maximal height of its elements), that is, the ends of $Axis(C_a)$ are not of full $\Z^{n-1}$-type. Hence, $a$ extends one of the ends of $Axis(C_a)$ to the full $\Z^{n-1}$-type.

\smallskip

(c) follows immediately from (a): every branch-point of $T - \Gamma_{K_T}$ is a $K_T$-translate of a branch-point of the intersection of $T$ with the union of the segments $[\varepsilon, x_S],\ S \in \mathcal{K}$. The union is finite, hence (c) is proved.

\smallskip

Now, let us prove (d). Let $a$ be an end from the finite set of representatives $\{a_1, \ldots, a_k\}$ of $K_T$-orbits of ends of $T - \Gamma_{K_T}$ chosen in part (b). As was proved in (b), $a$ extends one of the ends of $Axis(C_a)$ to the full $\Z^{n-1}$-type, where $C_a$ is a maximal abelian group of $K_T$. The axis $Axis(C_a)$ is an open interval $(\alpha, \beta)$ in $K_T$, where $\alpha$ and $\beta$ are ends of $K_T$ of $\Z^m$-type with $m < n - 1$. Assume that $a$ extends $\alpha$. Hence, $[x_T, a) -  \Gamma_{K_T}$ is an open interval $(\gamma, a)$, where $\gamma$ is an end of $T$ of $\Z^m$-type ``attached'' to $\alpha$. Note that if $C_a$ is non-trivial, then for any $g \in C_a$ we have $g \cdot \alpha = \alpha$ and $g \cdot a = a$, hence, $g \cdot (\gamma, a) = (\gamma, a)$. It follows that the translation along $Axis(C_a)$ induced by the action of $g$ propagates along $(\gamma, a)$ by the same distance. Hence, the label of $(\gamma, a)$ is periodic with multiple periods on various heights corresponding to finitely many generators of $C_a$. It follows that $(\gamma, a)$ can be relabeled by a disjoint copy of the alphabet $Y$.

Next, let $a, b \in \{a_1, \ldots, a_k\}$ be such that both $C_a$ and $C_b$ are non-trivial. If $[x_T, a) \cap [x_T, b)$ is not contained in $\Gamma_{K_T}$, then $a = b$. Indeed, since $[x_T, a) \cap [x_T, b)$ is not contained in $\Gamma_{K_T}$, the ends of $Axis(C_a)$ and $Axis(C_b)$ extended respectively by $a$ and $b$ must coincide in $\Gamma_{K_T}$. Hence, $C_a = C_b$ and in this case every $c \in C_a$ fixes both ends $a$ and $b$, which is possible only if $a = b$. If $[x_T, a) \cap [x_T, g \cdot a)$, where $g \in K_T$, is not contained in $\Gamma_{K_T}$, then $C_a = g C_a g^{-1}$, which implies that $g \in C_a$ and $g \cdot a = a$. Now, suppose $[x_T, a) \cap [x_T, g \cdot b)$, where $g \in K_T$, is not contained in $\Gamma_{K_T}$. Then, using the same argument as before, we obtain $C_a = g C_b g^{-1}$, but then $a = g \cdot b$, which is a contradiction since both $a$ and $b$ are representatives of different $K_T$-orbits of ends of $T - \Gamma_{K_T}$. Finally, if $[x_T, f \cdot a) \cap [x_T, h \cdot b)$, where $f, h \in K_T$, is not contained in $\Gamma_{K_T}$, then we obtain $f C_a f^{-1} = h C_b h^{-1}$ and $f \cdot a = h \cdot b$, which is again a contradiction. In other words, $K_T$-translates of any two ends $a, b \in \{a_1, \ldots, a_k\}$ that extend non-trivial centralizers of $K_T$ do not intersect outside of $\Gamma_{K_T}$. 

Finally, for each $a \in \{a_1, \ldots, a_k\}$ let $Y_a$ be a copy of the alphabet $Y$ so that $Y_a \cap Y_b = \emptyset$ if $a \neq b$. If $C_a$ is not trivial, then we relabel the open interval $[x_T, a) - \Gamma_{K_T}$ keeping required periodicity. Note that if $[x_T, a) \cap [x_T, b)$ is not contained in $\Gamma_{K_T}$, where $C_b$ is trivial, then $[x_T, a) \cap [x_T, b) = [x_T, y]$ and $[x_T, y] - \Gamma_{K_T}$ gets labeled by $Y_a$ as a part of $[x_T, a) - \Gamma_{K_T}$, while $[y, b)$ can be relabeled by $Y_b$ arbitrarily. 

Once the intersection of $T - \Gamma_{K_T}$ with the union of the segments $[\varepsilon, x_S],\ S \in \mathcal{K}$ is relabeled, one can $K_T$-equivariantly extend the new labeling from each $[x_T, a) - \Gamma_{K_T}$ to its images under the action of $K_T$.

\smallskip

The effectiveness part easily follows. Indeed, the centralizers of $K_T$ associated to the generators $\gamma_e^{\pm 1}, e \in \Psi_G$ are a part of the effective hierarchy for $G$. Every end $a$ of $T - \Gamma_{K_T}$ is $K_T$-conjugate to an end $b$, which is defined by some $\gamma_e^{\pm 1}$. Hence, if $a = g \cdot b$, then $C_a = g C_b g^{-1}$ and $C_b$ is a part of the effective hierarchy. Next, every segment in the finite union of the segments $[\varepsilon, x_S],\ S \in \mathcal{K}$ is represented by a computable function, hence, the intersection of $T$ with this union of segments can be found effectively. It follows that one can effectively find representatives of branch-points in $T - \Gamma_{K_T}$. Finally, in (d), if $Y$ is given effectively, then from the construction above it follows that $Y'$ is a disjoint union of finitely many copies of $Y$, so it can be found effectively as well.
\end{proof}

\begin{corollary}
\label{co:label}
If $G$ is a finitely generated subgroup of $CDR(\Z^n, X)$, where $X$ is an arbitrary alphabet, then there exists a finite alphabet $X'$ such that $G$ embeds into $CDR(\Z^n, X')$.
\end{corollary}
\begin{proof} 
Follows from Lemma \ref{le:trivial_stab} and Lemma \ref{le:stab1}.
\end{proof}

Note that we can assume that the underlying alphabet $X$ is finite (by virtue of Corollary \ref{co:label}), we construct a regular completion of $G$.

For a non-linear $\Z^{n-1}$-subtree $T$ of $\mathcal{K}$ with a non-trivial stabilizer, let $\mathcal{B}(T)$ be the set of representatives of branch-points of $T - \Gamma_{K_T}$. By Lemma \ref{le:stab1}, $\mathcal{B}(T)$ is finite and every branch-point of $T - \Gamma_{K_T}$ is $K_T$-equivalent to a branch-point from $\mathcal{B}(T)$. Let
$$\mathcal{D}(T) = \{ \mu([x_T,y]) \mid y \in \mathcal{B}(T) \}.$$
Observe that $\mathcal{D}(T)$ is a finite subset of $CDR(\Z^{n-1}, X)$.

\smallskip

Let $g \in G$. Hence, $[\varepsilon, g \cdot \varepsilon]$ meets finitely many $\Z^{n-1}$-subtrees $T_0, T_1, \ldots, T_k$, where $T_0$ is a $\Z^{n-1}$-subtree of $\mathcal{K}$ that contains $\varepsilon$ and $T_i$ is adjacent to $T_{i-1}$ for every $i \in [1, k]$. We have
$$[\varepsilon, g \cdot \varepsilon] \subseteq [x_{T_0}, x_{T_1}] \cup \cdots \cup [x_{T_{k-1}}, x_{T_k}] \cup [x_{T_k}, g \cdot \varepsilon].$$
Now, there exists $g_0 \in Stab_G(T_0)$ and a $\Z^{n-1}$-subtree $S_1$ of $\mathcal{K}$ adjacent to $T_0$ such that $T_1 = g_0 \cdot S_1$. Next, there exists $g_1 \in Stab_G(T_1)$ and a $\Z^{n-1}$-subtree $S_2$ of $\mathcal{K}$ adjacent to $S_1$ such that $T_2 = (g_1 g_0) \cdot S_2$, and so on. After $k$ steps we find a sequence of $\Z^{n-1}$-subtrees $S_0, S_1, \ldots, S_k$ from $\mathcal{K}$, where $S_0 = T_0, S_i\ {\rm is\ adjacent\ to}\ S_{i-1},\ i \in [1,k]$ and $T_i = (g_{i-1} \cdots g_0) \cdot S_i$, where $g_i \in Stab_G(T_i)$. Hence,
$$[\varepsilon, g \cdot \varepsilon] \subseteq [x_{T_0}, g_0 \cdot x_{T_0}] \cup [g_0 \cdot x_{T_0}, g_0 \cdot x_{S_1}] \cup [g_0 \cdot x_{S_1}, x_{T_1}] \cup [x_{T_1}, (g_1 g_0) \cdot x_{S_1}]$$
$$\cup [(g_1 g_0) \cdot x_{S_1}, (g_1 g_0) \cdot x_{S_2}] \cup \cdots \cup [(g_{k-1} \cdots g_0) \cdot x_{S_{k-1}}, (g_{k-1} \cdots g_0) \cdot x_{S_k}]$$
$$\cup [(g_{k-1} \cdots g_0) \cdot x_{S_k}, x_{T_k}] \cup [x_{T_k}, (g_k \cdots g_0) \cdot x_{S_k}],$$
where $(g_k \cdots g_0) \cdot x_{S_k} = g \cdot \varepsilon$.

Since
$$\mu([p,  q]) = \mu(g \cdot [p,  q]) = \mu([g \cdot p,  g \cdot q])$$
and
$$[(g_{i-1} \cdots g_0) \cdot x_{S_{i-1}}, (g_{i-1} \cdots g_0) \cdot x_{S_i}] = (g_{i-1} \cdots g_0) \cdot [x_{S_{i-1}}, x_{S_i}]$$
for $i \in [1,k]$, we obtain
$$\mu([(g_{i-1} \cdots g_0) \cdot x_{S_{i-1}}, (g_{i-1} \cdots g_0) \cdot x_{S_i}]) = \mu([x_{S_{i-1}}, x_{S_i}]).$$
Also, observe that for any $i \in [1,k]$
$$[(g_{i-1} \cdots g_0) \cdot x_{S_i}, x_{T_i}] \cup [x_{T_i}, (g_i \cdots g_0) \cdot x_{S_i}]$$
is a path in $T_i$, where $(g_{i-1} \cdots g_0) \cdot x_{S_i}$ and $(g_i \cdots g_0) \cdot x_{S_i}$ are $Stab_G(T_i)$-equivalent to $x_{T_i}$. So, it follows that
$$\mu([x_{T_i}, (g_{i-1} \cdots g_0) \cdot x_{S_i}]) = f_i \in K_{T_i},\ \ \mu([x_{T_i}, (g_i \cdots g_0) \cdot x_{S_i}]) = h_i \in K_{T_i}.$$
Next, note that $g_0 = \mu([x_{T_0}, g_0 \cdot x_{T_0}])$. Eventually, we have
$$g = g_0 \ast c_{S_0, S_1} \ast (f_1^{-1} \ast h_1) \ast c_{S_1, S_2} \ast \cdots \ast c_{S_{k-1}, S_k} \ast (f_k^{-1} \ast h_k),$$
where $c_{S_{i-1}, S_i}$ is the label of the path $[x_{S_{i-1}}, x_{S_i}]$ and the product on the right-hand side is defined in $CDR(\Z^n,X)$.

\smallskip

Now we are ready to perform the induction step.

\begin{theorem}
\label{th:completion}
Let $G$ be a finitely generated subgroup of $CDR(\Z^n, X)$ and assume that $\mathcal{K},\ \Psi_G,\ \Delta_G$ etc. are defined for $G$ as above. Suppose that for every non-linear $\Z^{n-1}$-subtree $T$ of $\mathcal{K}$ with a non-trivial stabilizer there exist
\begin{enumerate}
\item[(a)] an alphabet $Y(T)$,
\item[(b)] a $\Z^{n-1}$-tree $T'$, whose edges are labeled by $Y(T)$, and
\item[(c)] a finitely generated group $H_T \subset CDR(\Z^{n-1}, Y(T))$,
\end{enumerate}
such that $\Gamma_{H_T}$ is embedded into $T'$ and the action of $H_T$ on $\Gamma_{H_T}$ extends to a regular action of $H_T$ on $T'$. Moreover, assume that there is an embedding $\psi_T: T \rightarrow T'$, where $\psi_T(\Gamma_{K_T}) \subseteq \Gamma_{H_T}$, which induces an embedding $\phi_T : K_T \rightarrow H_T$, and such that if $a$ and $b$ are non-$K_T$-equivalent ends of $T$, then $\psi_T(a)$ and $\psi_T(b)$ are non-$H_T$-equivalent ends of $\psi_T(T)$.

Then there exists an embedding of $\displaystyle \bigcup_{T \in \mathcal{K}} \mathcal{D}(T)$ into $CDR(\Z^n, Y)$, where $Y$ is a finite alphabet containing $\displaystyle \bigcup_{T \in \mathcal{K}} Y(T)$, such that
\begin{enumerate}
\item[(i)] the union $\displaystyle \bigcup_{T \in \mathcal{K}} \left( H_T \cup \mathcal{D}(T) \cup \{c_{x_T, x_S} \mid S\ {\rm is\ adjacent\ to}\ T\ {\rm in}\ \mathcal{K} \} \right)$ generates a group $H \subset CDR(\Z^n, Y)$, which acts regularly on $\Gamma_H$ with respect to $\varepsilon_H$,
\item[(ii)] there exists an embedding $\psi: \Gamma_G \rightarrow \Gamma_H,\ \psi(\varepsilon_G) = \varepsilon_H$, which induces an embedding $\phi : G \rightarrow H$, such that if $a$ and $b$ are non-$G$-equivalent ends of $\Gamma_G$, then $\psi(a)$ and $\psi(b)$ are non-$H$-equivalent ends of $\psi(\Gamma_G)$.
\end{enumerate}
Moreover, if $G$ has an effective hierarchy over $X$ and for every non-linear $\Z^{n-1}$-subtree $T$ of $\mathcal{K}$ with a non-trivial stabilizer, the group $H_T$ has an effective hierarchy over $Y(T)$, then $H$ has an effective hierarchy over $Y$.
\end{theorem}
\begin{proof} First of all, by Corollary \ref{co:label} we can assume $X$ to be finite. Hence, we can assume that any two distinct $\Z^{n-1}$-subtrees $S$ and $T$ of $\mathcal{K}$ are labeled by distinct alphabets $X(S)$ and $X(T)$. Next, by Lemma \ref{le:trivial_stab}, in each $\Z^{n-1}$-subtree $S$ of $\mathcal{K}$ with trivial stabilizer there are only finitely many
branch-points, so we can cut $S$ along these branch-points, obtain finitely many closed and half-open segments, and relabel them by a finite alphabet. Thus, we can assume all this to be done already.

\smallskip

Let $T$ be a non-linear $\Z^{n-1}$-subtree of $\mathcal{K}$ with a non-trivial stabilizer. Observe that by Lemma \ref{le:stab1}, every end $a$ of $T$ either is an end of $\Gamma_{K_T}$, or $a = g \cdot a_0$, where $a_0$ belongs to a finite list of representatives of orbits of ends of $T - \Gamma_{K_T}$.

By the assumption, $T$ embeds into $T'$ labeled by $Y(T)$, while $\Gamma_{K_T}$ embeds into $\Gamma_{H_T}$, where $H_T$ acts regularly on $T'$. It follows that for every branch-point $b$ of $T$, the label of $\psi_T([x_T, b])$ defines an element of $H_T$. In particular, the label of $\psi_T(d)$ belongs to $H_T$ for every $d \in \mathcal{D}(T)$. Moreover, if $S_1, S_2$ are $\Z^{n-1}$-subtrees of $\mathcal{K}$ adjacent to $T$ and $a_{S_1}, a_{S_2}$ are the corresponding ends of $T$, then $a_{S_1}$ and $a_{S_2}$ are non-$K_T$-equivalent ends of $T$. Hence, by the assumption, $\psi_T(a_{S_1})$ and $\psi_T(a_{S_2})$ are non-$H_T$-equivalent ends of $T'$, and it follows that
$$(h_1 \cdot \psi_T([x_T, a_{S_1})) \cap (h_2 \cdot \psi_T([x_T, a_{S_2}))$$
is a closed segment of $T'$ for every $h_a, h_2 \in H_T$. Hence,
$$com(h_1 \ast c_{x_T, x_{S_1}}, h_2 \ast c_{x_T,x_{S_2}}),$$
is defined in $CDR(\Z^{n-1}, Y(T))$. Since $X(T) \cap X(S) = \emptyset$, we have that $h \ast c^{-1}_{x_T, x_S} = h \circ c^{-1}_{x_T, x_S}$ for every $\Z^{n-1}$-subtree $S$ of $\mathcal{K}$ adjacent to $T$. Thus, the union
$$H_T \cup \mathcal{D}(T) \cup \{ c_{x_T, x_S} \mid S\ {\rm is\ adjacent\ to}\ T\ {\rm in}\ \mathcal{K} \}$$
generates a subgroup $H'_T$ in $CDR(\Z^n, Q)$, where
$$Q = \bigcup_{T \in \mathcal{K}} Y(T),$$
so that $T$ embeds into $\Gamma_{H'_T}$. Moreover, $H'_T$ acts regularly on $\Gamma_{H'_T}$.

\smallskip

Now, from the fact that $X(T) \cap X(S) = \emptyset$ if $T$ is not $G$-equivalent to $S$, it follows that $\displaystyle \bigcup_{T \in \mathcal{K}} H'_T$ generates a subgroup $H$ of $CDR(\Z^n, Y)$, where $Y$ is a finite alphabet containing $Q$. Observe that $\Gamma_{H'_T}$ embeds into $\Gamma_H$ for each $T \in \mathcal{K}$. Moreover, for every $f, g \in H$ we have that $w =
Y(\varepsilon_H,\ f \cdot \varepsilon_H,\ g \cdot \varepsilon_H)$ belongs to $\Gamma_{H'_T}$ for some $T \in \mathcal{K}$, hence, the label of the segment $[\varepsilon_H, w]$ defines an element of $H'_T \subset H$. That is, $H$ acts regularly on $\Gamma_H$.

Next, since
$$G \leqslant \langle K_T,\ \{ \mathcal{D}(T) \mid T \in \mathcal{K} \} \rangle \leqslant H,$$
it follows that $G$ embeds into $H$.

\smallskip

Finally, every end $a$ of $\Gamma_G$ uniquely corresponds to an end in $\Delta_G$. Every end of $\Delta_G$ can be viewed as a reduced infinite path $p_a$ in $\Delta_G$ originating at $v \in \Delta_G$, which is the image of $\varepsilon \in \Gamma_G$. Observe that two ends $a$ and $b$ of $\Gamma_G$ are $G$-equivalent if and only if $\pi(p_a) = \pi(p_b)$ in $\Psi_G$ (recall that $\pi : \Delta_G \rightarrow \Delta_G / G = \Psi_G$).

Denote $\Delta_H = \Gamma_H / \sim$, where ``$\sim$'' is the equivalence of $\Z^{n-1}$-close points. Since $\psi : \Gamma_G \rightarrow \Gamma_H$ is an embedding, $\Delta_G$ embeds into $\Delta_H$ and, with abuse of notation, we are going to denote this embedding by $\psi$ again. Let $w = \psi(v)$.

Let $a$ and $b$ be non-$G$-equivalent ends of $\Gamma_G$ and let
$$p_a = v\ v_1\ v_2 \cdots,\ \ p_b = v\ u_1\ u_2 \cdots.$$
Assume that $\psi(a)$ and $\psi(b)$ are $H$-equivalent in $\Gamma_H$, that is, there exists $h \in H$ such that $h \cdot p_{\psi(a)} = p_{\psi(b)}$. Since $p_{\psi(a)}$ and $p_{\psi(b)}$ have the same origin $w$, it follows that $h \cdot w = w$, that is, $h \in Stab_H(T'_0)$, where $T'_0$ is the $\Z^{n-1}$-subtree of $\Gamma_H$ containing $\psi(T_0)$. Moreover, if $e_1 = (w, \psi(v_1)),\
f_1 = (w, \psi(u_1))$, then $h \cdot e_1 = f_1$ and it follows that $h \cdot a_1 = b_1$, where $a_1$ and $b_1$ are the ends of $\psi(T_0)$ corresponding to $e_1$ and $f_1$. By the assumption of the theorem, there exists $\phi(g_1) \in Stab_{\phi(G)}(\psi(T_0))$ such that $\phi(g_1) \cdot a_1 = b_1$, so, $\phi(g_1) \cdot \psi(v_1) = \psi(u_1)$. Since $\phi : G \rightarrow H$ and $\psi : \Gamma_G \rightarrow \Gamma_H$ are embeddings, it follows that $g_1 \cdot v_1 = u_1$ and $\pi(u_1) = \pi(v_1)$ in $\Delta_G$.

Continuing in the same way we obtain $\pi(u_i) = \pi(v_i),\ i \geqslant 1$ in $\Delta_G$, so, $a$ and $b$ are $G$-equivalent. Hence, 
$\psi(a)$ and $\psi(b)$ are $H$-equivalent, which is a contradiction with our assumption.

\smallskip

The effectiveness part follows from the effectiveness parts of Lemma \ref{le:trivial_stab} and Lemma \ref{le:stab1}.
\end{proof}

\begin{theorem}
\label{co:main}
Let $G$ be a finitely generated subgroup of $CDR(\Z^n, X)$, where $X$ is arbitrary. Then there exists a finite alphabet $Y$ and an embedding $\phi : G \rightarrow H$, where $H$ is a finitely generated subgroup of $CDR(\Z^n, Y)$ with a regular length function, such that $|g|_X = |\phi(g)|_Y$ for every $g \in G$. Moreover, if $G$ has an effective hierarchy over $X$, then $H$ has an effective hierarchy over $Y$
\end{theorem}
\begin{proof} 
We use the induction on $n$. If $n = 1$, then the result follows from Lemma \ref{le:n=1}. Finally, the induction step follows from Theorem \ref{th:completion}. 
\end{proof}

\end{document}